\documentclass[reqno,12pt,letterpaper]{amsart}
\usepackage{amsmath,amssymb,amsthm,graphicx,mathrsfs,url}
\usepackage[usenames,dvipsnames]{color}
\usepackage[colorlinks=true,linkcolor=Red,citecolor=Green]{hyperref}

\def\?[#1]{\textbf{[#1]}\marginpar{\Large{\textbf{??}}}}
\newtheorem{prop}{Proposition}
\newtheorem{thm}[prop]{Theorem}

\newtheorem{lem}[prop]{Lemma}
\newtheorem{cor}[prop]{Corollary}

\numberwithin{equation}{section}
\numberwithin{prop}{section}

\renewcommand{\Re}{\mathop{\rm Re}\nolimits}
\renewcommand{\Im}{\mathop{\rm Im}\nolimits}

\DeclareMathOperator{\comp}{comp}
\DeclareMathOperator{\Op}{Op}
\DeclareMathOperator{\supp}{supp}
\DeclareMathOperator{\Sp}{Sp}
\DeclareMathOperator{\Span}{Span}

\DeclareMathOperator{\WF}{WF}

\begin{document}
\title{Damped wave equations on compact hyperbolic surfaces}
\author{Long Jin}
\email{long249@purdue.edu}
\address{Department of Mathematics, Purdue University,
150 N. University St, West Lafayette, IN 47907}

\begin{abstract} 
We prove exponential decay of energy for solutions of the damped wave equation on compact hyperbolic surfaces with regular initial data as long as the damping is nontrivial. The proof is based on a similar strategy as in \cite{measupp} and in particular, uses the fractal uncertainty principle proved in \cite{fullgap}.
\end{abstract}

\maketitle

\section{Introduction}
\label{s:intro}

In this paper, we always let $M$ be a compact (connected) hyperbolic surface (with constant negative curvature $-1$) and $\Delta$ be the Laplace-Beltrami operator on $M$. We investigate the long time behavior of the damped wave equation on $M$ with damping function $a\in C^\infty(M)$ such that $a\geq0$ but $a\not\equiv0$:
\begin{equation}
	\label{e:dampwave}
(\partial_t^2-\Delta+2a(x)\partial_t)v(t,x)=0,
\quad v|_{t=0}=v_0(x), \partial_tv|_{t=0}=v_1(x).
\end{equation}
For initial data $(v_0,v_1)\in \mathcal{H}:=H^1(M)\times L^2(M)$, we consider the energy of the solution
\begin{equation}
	\label{e:energy}
E(v(t)):=\frac{1}{2}\int_M|\partial_tv(t,x)|^2+|\nabla v(t,x)|^2dx.
\end{equation}

Our main theorem is the exponential decay of the energy for solutions to \eqref{e:dampwave} with regular initial conditions.
\begin{thm}
	\label{t:energy-decay}
For every $s>0$, there exist constants $C$ and $\gamma=\gamma(s)>0$ such that for any $(v_0,v_1)\in \mathcal{H}^s=H^{s+1}(M)\times H^s(M)$, we have exponential decay of the energy:
\begin{equation}
	\label{e:energy-decay}
E(v(t))\leq Ce^{-\gamma t}\|(v_0,v_1)\|_{\mathcal{H}^s}^2.
\end{equation}
\end{thm}


\subsection{Eigenvalue problem}
\label{s:eigen}
The decay of the energy is closely related to the spectrum of the operator
\begin{equation}
	\label{e:B-matrix}
\mathcal{B}=\begin{pmatrix}
0 & I\\
-\Delta & -2ia
\end{pmatrix}:\mathcal{D}(\mathcal{B})=H^2\times H^1\to H^1\times L^2.
\end{equation}
The strongly continuous semigroup $e^{-it\mathcal{B}}$, $t\geq0$ maps $(v_0,iv_1)\in\mathcal{H}$ to $(v(t),i\partial_tv(t))$ where $v$ is the solution of the damped wave equation \eqref{e:dampwave}.


It is well known that (see e.g. \cite{lebeau}) the spectrum of $\mathcal{B}$ consists of a discrete sequence of eigenvalues with each eigenspace finite dimensional. For $\tau\in\Sp(\mathcal{B})$, there is $u\in H^1$ such that $v(t,x)=e^{-it\tau}u(x)$ satisfies the equation \eqref{e:dampwave} and thus $u$ is an eigenfunction of the (nonlinear) eigenvalue problem
\begin{equation}
	\label{e:eigen}
P(\tau)u:=(-\Delta-\tau^2-2ia\tau)u=0.
\end{equation}

The spectrum $\Sp(\mathcal{B})$ is symmetric with respect to the imaginary axis and is contained in $\{-2\|a\|_\infty\leq\Im\tau\leq0\}$. Moreover the only real eigenvalue is $\tau=0$ which is simple with eigenfunctions being constant functions. If $\Re\tau\neq0$, we have $-\|a\|_\infty\leq\Im\tau<0$, (see \cite{nonnenmacher}).

For any $T>0$, we define a function on the unit cosphere bundle $S^\ast M$ by
\begin{equation}
	\label{e:averagedamp}
\langle a\rangle_T(x,\xi)=\frac{1}{T}\int_0^T\pi^\ast a\circ\varphi_t(x,\xi)dt,
\end{equation}
where $\pi:S^\ast M\to M$ is the natural projection and $\varphi_t$ is the geodesic flow on $S^\ast M$ (see Section \ref{s:dynamic}). We also define the minimal and maximal asymptotic damping constants to be
\begin{equation}
	\label{e:asydamp}
a_-:=\sup_{T>0}\inf_{S^\ast M}\langle a\rangle_T=\lim_{T\to\infty}\inf_{S^\ast M}\langle a\rangle_T,\quad
a_+:=\inf_{T>0}\sup_{S^\ast M}\langle a\rangle_T=\lim_{T\to\infty}\sup_{S^\ast M}\langle a\rangle_T.
\end{equation}
Then there are only finite number of eigenvalues of $\mathcal{B}$ with $\Im\tau\not\in(-a_+-\varepsilon,-a_-+\varepsilon)$ for any $\varepsilon>0$ (see Lebeau \cite{lebeau}). Note that if there is a closed geodesic not intersecting $\supp a$, then $a_-=0$.


The main step to prove Theorem \ref{t:energy-decay} is to establish a spectral gap for $\mathcal{B}$ as $\Re\tau\to\infty$.

\begin{thm}
	\label{t:gap}
There exists $\beta,C_0>0$ such that for any $\tau\in\Sp(\mathcal{B})$ with $|\Re\tau|\geq C_0$, we have $\Im\tau<-\beta$.
In particular, 
\begin{equation}
	\label{e:gap}
G:=\inf\{-\Im\tau:\tau\in\Sp(\mathcal{B})\setminus\{0\}\}>0.
\end{equation}
\end{thm}

We call $G$, defined in \eqref{e:gap}, the spectral gap of the eigenvalue problem \eqref{e:eigen} (or the damped wave equation \eqref{e:dampwave}) and $\beta$ in Theorem \ref{t:gap} the essential spectral gap. In particular, there are no eigenvalues $\tau\in\Sp(\mathcal{B})$ with $-G<\Im\tau<0$ and only finite number of eigenvalues $\tau\in\Sp(\mathcal{B})$ with $-\beta<\Im\tau<0$.

In the general spirit of \cite{csvw}, we can obtain Theorem \ref{t:energy-decay} on the decay of energy for damped wave equation \eqref{e:dampwave} from a suitable resolvent estimate when $|\Re\tau|\geq C_0$ and $\Im\tau>-\beta$, see Theorem \ref{t:resolvent}. 

We also remark that we can obtain an eigenvalue expansion, similar to \cite[Theorem 2]{schenckpressure}, with eigenvalues satisfying $\Im\tau>-\beta$, by the same argument there. Here $\beta$ is the essential spectral gap given in Theorem \ref{e:gap}, see also \cite{hitrik}. We refer to these references for details.

\subsection{Previous results}
\label{s:results}

On a general compact manifold, the solution to the equation \eqref{e:dampwave} stabilizes by any damping: for any $(v_0,v_1)\in\mathcal{H}$, $E(v(t))\to0$ as $t\to+\infty$, (see e.g. \cite{nonnenmacher}). However, uniform exponential decay of energy for all $(v_0,v_1)\in\mathcal{H}$ is equivalent to the geometric control condition (see e.g. \cite{blr}) for $\Omega=\{a>0\}$:
\begin{equation}
	\label{e:geocontrol}
\text{There exists } L=L(M,\Omega)>0 \text{ s.t. every geodesic of length } L \text{ on } M \text{ intersects } \Omega.
\end{equation}
(Or equivalently, with the notation \eqref{e:averagedamp}, $\langle a\rangle_L>0$ everywhere for some $L>0$.) This was first proved by Rauch--Taylor \cite{rauchtaylor} in various settings, see also Bardos--Lebeau--Rauch \cite{blr}, Lebeau \cite{lebeau}, and  Hitrik \cite{hitrik}. In particular, Lebeau \cite{lebeau} determined the optimal exponential decay rate
\begin{equation*}
\gamma_{\max}:=\sup\left\{\gamma\geq0\mid \exists C>0, \text{ such that } \forall (v_0,v_1)\in\mathcal{H}, E(v(t))\leq Ce^{-\gamma t}E(v(0))\right\}.
\end{equation*}
to be 
\begin{equation*}
\gamma_{\max}=2\min(G, a_-)
\end{equation*}
where $G$ is the spectral gap defined in \eqref{e:gap} and $a_-$ is the minimal asymptotic damping constant defined in \eqref{e:asydamp}. Note that the positivity of the spectral gap $G$ is not enough to ensure the uniform exponential decay, as shown by an example in Lebeau \cite{lebeau} for which $G>0$ but $\gamma_{\max}=a_-=0$. (See also \cite{renardy}.) 

When the geometric control condition \eqref{e:geocontrol} fails, there are no uniform decay for initial data in $\mathcal{H}$. But for regular initial data in $\mathcal{H}^s$ for some $s>0$, it is possible to obtain some uniform decay. In the most general situation, Lebeau \cite{lebeau} showed that there is a uniform logarithmic decay of energy for initial data in $\mathcal{H}^s$ which is also optimal as shown by an example with an elliptic closed geodesic not passing $\{a>0\}$ in the same paper. 

For special manifolds, we may get better decay rates for initial data in $\mathcal{H}^s$ with $s>0$. For example, Anantharaman--L\'{e}autaud \cite{nalinidamped} proved a sharp polynomial decay rate on tori (or square), see also earlier work of Liu--Rao \cite{liurao} and Phung \cite{phung}. Other examples with polynomial decay rates were shown by Burq--Hitrik \cite{burqhitrik}, Burq--Zuily \cite{burqzuily}, Burq--Zworski \cite{bz04}, Christianson--Schenck--Vasy--Wunsch \cite{csvw} and L\'{e}autaud--Lerner \cite{ll}. When the ``undamped set'' is a single hyperbolic closed orbit, Chrisitianson \cite{hans,hanscorrection} showed that there is a subexponential decay which is also optimal in general as shown by an example of Burq--Christianson \cite{burqchr}. This has been further generalized to the situation where the ``undamped set'' is normally hyperbolic by Christianson--Schenck--Vasy--Wunsch \cite{csvw} and hyperbolic with small pressure in the same paper and by Nonnenmancher--Rivi\`ere \cite[Appendix]{nore}.

On manifolds with negative curvature, Schenck \cite{schenckpressure, schenck} proved exponential decay for initial data in $\mathcal{H}^s$ with $s>0$ under a pressure condition which holds if the ``undamped set'' is thin and the damping is strong enough (see also Nonnenmacher \cite{nonnenmacher} for a condition on the ``least damped set''). Our result removes this condition for compact hyperbolic surfaces. 

The geodesic flows on negatively curved manifolds are chaotic, in particular, Anosov, so the damped wave equation can be viewed as an example of \emph{damped quantum chaotic} system. Both the results of Schenck and the present paper use techniques from the study of parallel \emph{closed} or \emph{open} quantum chaotic systems. In particular, Schenck uses the \emph{hyperbolic dispersive estimates} from the work of Nonnenmacher--Zworski \cite{nozw} on the pressure gaps of general open quantum chaotic systems, which is based on the ideas from earlier work of Anantharaman--Nonnenmancher \cite{anno} and Anantharaman \cite{anan} proving lower bounds on entropy of semiclassical measures for Laplacian eigenfunctions on Anosov manifolds. 

In our paper, we adapt the approach from a joint work with Dyatlov \cite{measupp} which shows that semiclassical measures for Laplacian eigenfunctions on compact hyperbolic surfaces have full support. The key idea in \cite{measupp} is a new approach called \emph{fractal uncertainty principle} developed by Dyatlov--Zahl \cite{hgap} and Bourgain--Dyatlov \cite{fullgap} to obtain \emph{essential spectral gaps} for \emph{convex co-compact hyperbolic surfaces} when the pressure condition fails. We refer to the papers above and references there for a detailed discussion of the fractal uncertainty principle and other aspects of quantum chaotic systems.

Finally, we mention that the eigenvalues of $\mathcal{B}$ satisfy a Weyl law which is proved by Markus--Matsaev \cite{mama}, see also Sj\"{o}strand \cite{sjostrand}. Sj\"{o}strand \cite{sjostrand} also establish a concentration result on the imaginary parts of eigenvalues in a potentially smaller strip than $\Im\tau\in(-a_-,-a_+)$. In the case of manifolds with ergodic geodesic flow (which is true for compact hyperbolic surfaces), Sj\"{o}strand's result states that most eigenvalues are near the line $\Im\tau=-\langle a\rangle_M$ where $\langle a\rangle_M$ is the average of $a$ over $M$. Anantharaman \cite{nalinideviation} refined this statement by showing a deviation result on the imaginary parts of eigenvalues. We refer to these references for a further discussion of the the distribution of eigenvalues of $\mathcal{B}$.

\subsection{Organization of the paper}
The paper is organized as follows. In Section \ref{s:prelims}, we review some basic facts about hyperbolic surfaces and semiclassical analysis, especially the exotic symbol calculus (section \ref{s:fancy-calculus}) developped in \cite{hgap} and \cite{measupp}. Then we formulate Theorem \ref{t:propagator-decay} about the decay of a general \emph{semiclassical damped propagator} localizing near the energy surface. In Section \ref{s:proof-propagator}, we prove Theorem \ref{t:propagator-decay} using a similar strategy as in \cite{measupp}. Roughly speaking, we separate the energy surface into the ``damped'' part and the ``undamped'' part. The \emph{damped} part has a natural decay while the \emph{undamped} part is ``fractal'' which allows us to use the \emph{fractal uncertainty principle} to obtain the decay. Finally in Section \ref{s:proof}, we prove Theorem \ref{t:energy-decay} and \ref{t:gap} by establishing a polynomial resolvent bound and adapting the standard arguments.

\subsection*{Acknowledgement}
I am very grateful to Kiril Datchev, Semyon Dyatlov and Maciej Zworski for the encouragement to work on this project and many great suggestions to the early draft of the paper. I would also like to thank Hans Christianson and Jared Wunsch for the helpful discussions about damped wave equations. Part of the work is done during my visit to YMSC at Tsinghua University for which I appreciate the hospitality.

\section{Preliminaries}
  \label{s:prelims}
  
In this part, we review some basic setup as in \cite[\S2]{measupp} as well as some modification we need to study the damped wave equation \eqref{e:dampwave}.

\subsection{Dynamics of the geodesic flows on hyperbolic surfaces}
	\label{s:dynamic}

Let $(M,g)$ be a compact hyperbolic surface and $T^*M\setminus 0$ denote the cotangent bundle $(x,\xi)\in T^*M$ with the zero section removed. Let $p\in C^\infty(T^*M\setminus 0;\mathbb R)$ be defined by
\begin{equation}
	\label{e:p-def}
p(x,\xi)=|\xi|_g.
\end{equation}
Then the homogeneous geodesic flow is the Hamiltonian flow of $p$,
\begin{equation}
	\label{e:p-flow}
\varphi_t:=\exp(tH_p):T^*M\setminus 0\to T^*M\setminus 0.
\end{equation}
We also write $S^\ast M=p^{-1}(1)$ to be the unit cosphere bundle.

From now on, we always assume that $M$ is orientable; if not, we may pass to a double cover of $M$. For the definition of the anisotropic calculi later (Section \ref{s:fancy-calculus}), we need the following notation for the weak stable/unstable spaces:
\begin{equation}
	\label{e:l-foliations}
L_s:=\Span(H_p,U_+),\quad
L_u:=\Span(H_p,U_-)\
\subset\
T(T^*M\setminus 0).
\end{equation}
so that $L_s,L_u$ are Lagrangian foliations, see~\cite[Lemma~4.1]{hgap}. Here we use an explicit frame on $T^*M\setminus 0$ consisting of four vector fields
\begin{equation}
	\label{e:canonical-fields}
H_p,U_+,U_-,D
\in C^\infty\big(T^*M\setminus 0;T(T^*M\setminus 0)\big),
\end{equation}
where $H_p$ is the generator of $\varphi_t$, $D=\xi\cdot\partial_\xi$
is the generator of dilations and the vector fields $U_\pm$
are defined on $S^*M$ as stable/unstable horocyclic vector fields and extended homogeneously to $T^*M\setminus 0$, so that the following commutation relations hold:
\begin{equation}
	\label{e:comm-rel}
[U_\pm,D]=[H_p,D]=0,\quad [H_p,U_\pm]=\pm U_\pm.
\end{equation}

\subsection{Operators and propagation}
  \label{s:prelim-opers}

We first briefly review the standard classes of semiclassical pseudodifferential operators with classical symbols $\Psi^k_h(M)$. We refer the reader to the book of Zworski~\cite{e-z} for an introduction to semiclassical analysis used in this paper, to~\cite[\S14.2.2]{e-z} for pseudodifferential operators on manifolds,
and to~\cite[\S E.1.5]{resonance} and ~\cite[\S2.1]{hgap} for the classes $\Psi^k_h(M)$ used here. The corresponding symbol classes is denoted by $S^k(T^*M)$ and we have the principal symbol map and a (non-canonical) quantization map
$$
\sigma_h:\Psi^k_h(M)\to S^k(T^*M),\quad
\Op_h:S^k(T^*M)\to \Psi^k_h(M).
$$
We also write $\Psi^{\comp}_h(M)$ to be the space of operators $A\in\Psi^k_h(M)$ with the wavefront set $\WF_h(A)$ being a compact subset of $T^*M$.

Applying sharp G\aa rding inequality (see \cite[Theorem 4.32]{e-z}) to the operator $I-A^\ast A$, we have the following $L^2$-norm bound on pseudodifferential operators:
\begin{equation}
	\label{e:basic-norm-bound}
A\in\Psi^0_h(M),\quad
\sup|\sigma_h(A)|\leq 1
\quad\Longrightarrow\quad
\|A\|_{L^2\to L^2}\leq 1+Ch.
\end{equation}

The operator $-h^2\Delta$ lies in $\Psi^2_h(M)$ and $\sigma_h(-h^2\Delta)=p^2$ with $p$ defined in~\eqref{e:p-def}. As in \cite{measupp}, we use an operator $P\in\Psi^{\comp}_h$ with principal symbol $p$ near the energy surface $S^\ast M$ instead of $-h^2\Delta$ for convenience. We fix a function
\begin{equation}
	\label{e:psi-P}
\psi_P\in C_0^\infty((0,\infty);\mathbb R),\quad
\psi_P(\lambda)=\sqrt{\lambda}\quad\text{for }\frac{1}{16}\leq \lambda\leq 16,
\end{equation}
and define the operator
\begin{equation}
	\label{e:the-P}
P:=\psi_P(-h^2\Delta),\quad
P^*=P.
\end{equation}

By the functional calculus of pseudodifferential operators,
see~\cite[Theorem~14.9]{e-z} or~\cite[\S8]{DimassiSjostrand}, we have
\begin{equation}
	\label{e:the-P-2}
P\in\Psi^{\comp}_h(M),\quad
\sigma_h(P)=p\quad\text{on }\{1/4\leq |\xi|_g\leq 4\}.
\end{equation}
The flow $\varphi_t$ is quantized (at least near $S^\ast M$) by the unitary propagator 
\begin{equation}
	\label{e:U-0-t}
U_0(t):=\exp(-itP/h):L^2(M)\to L^2(M).
\end{equation}

For a bounded operator $A:L^2(M)\to L^2(M)$, we use the notation
\begin{equation}
	\label{e:A-t}
A(t):=U_0(-t)A U_0(t).
\end{equation}
If $A\in \Psi^{\comp}_h(M)$, $\WF_h(A)\subset \{1/4<|\xi|_g<4\}$,
and $t$ is bounded uniformly in $h$, then Egorov's theorem~\cite[Theorem~11.1]{e-z} implies that
\begin{equation}
	\label{e:basic-egorov}
A(t)\in \Psi^{\comp}_h(M);\quad
\sigma_h(A(t))=\sigma_h(A)\circ\varphi_t.
\end{equation}

\subsection{Anisotropic calculi and long time propagation}
  \label{s:fancy-calculus}
  
To handle the propagation up to logarithmic time, we need a general calculus introduced in \cite[\S 3]{hgap} and in particular, the version developed in \cite[Appendix]{measupp}. Here we briefly review the definition and some basic properties and refer the reader to the references above for the details.

Fix $\rho\in [0,1)$, $\rho'\in[0,\rho/2]$ such that $\rho+\rho'<1$ and let $L\in \{L_u,L_s\}$ where the Lagrangian foliations $L_u,L_s$ are defined in~\eqref{e:l-foliations}. Define the class of $h$-dependent symbols
$S^{\comp}_{L,\rho,\rho'}(T^*M\setminus 0)$ as follows:
$a\in S^{\comp}_{L,\rho,\rho'}(T^*M\setminus 0)$ if
\begin{enumerate}
\item $a(x,\xi;h)$ is smooth in $(x,\xi)\in T^*M\setminus 0$,
defined for $0<h\leq 1$, and supported in an $h$-independent compact
subset of $T^*M\setminus 0$;
\item $a$ satisfies the derivative bounds
\begin{equation}
	\label{e:symbol-derby}
\sup_{x,\xi}|Y_1\ldots Y_mZ_1\ldots Z_k a(x,\xi;h)|\leq Ch^{-\rho k-\rho'm},\quad
0<h\leq 1
\end{equation}
for all vector fields
$Y_1,\dots,Y_m,Z_1,\dots,Z_k$ on $T^*M\setminus 0$ such that
$Y_1,\dots,Y_m$ are tangent to $L$.
Here the constant $C$ depends on $Y_1,\dots,Y_m$, $Z_1,\dots,Z_k$, and~$\varepsilon$ but does not depend on $h$.
\end{enumerate}
Moreover, we introduce the
class $S^{\comp}_{L,\rho}(T^*M\setminus 0)$ for $\rho\in[0,1)$ given by
\begin{equation*}
S^{\comp}_{L,\rho}(T^*M\setminus 0)=\bigcap_{\varepsilon>0} S^{\comp}_{L,\rho+\varepsilon,\varepsilon}(T^*M\setminus 0).
\end{equation*}

In terms of the frame~\eqref{e:canonical-fields}, the derivative
bounds~\eqref{e:symbol-derby} become
\begin{equation*}
\begin{split}
\sup_{x,\xi} \big|H_p^k U_+^\ell U_-^m D^n a(x,\xi;h)|&=\mathcal O(h^{-\rho(m+n)-\rho'(k+\ell)})\quad\text{for }L=L_s,\\
\sup_{x,\xi} \big|H_p^k U_-^\ell U_+^m D^n a(x,\xi;h)|&=\mathcal O(h^{-\rho(m+n)-\rho'(k+\ell)})\quad\text{for }L=L_u.
\end{split}
\end{equation*}

If $a\in C_0^\infty(T^*M\setminus 0)$ is an $h$-independent symbol, then it follows from the commutation relations~\eqref{e:comm-rel} that
$$
H_p^k U_+^\ell U_-^m D^n(a\circ\varphi_t)=e^{(m-\ell)t}(H_p^k U_+^\ell U_-^m D^na)\circ\varphi_t.
$$
Therefore
\begin{equation*}
\begin{split}
a\circ\varphi_t\in S^{\comp}_{L_s,\rho}(T^*M\setminus 0)\quad\text{uniformly in }
t,\quad
0\leq t\leq \rho\log(1/h)\\
a\circ\varphi_{-t}\in S^{\comp}_{L_u,\rho}(T^*M\setminus 0)\quad\text{uniformly in }
t,\quad
0\leq t\leq \rho\log(1/h).
\end{split}
\end{equation*}

Let $\Psi^{\comp}_{h,L,\rho,\rho'}(M)$ and $\Psi^{\comp}_{h,L,\rho}(M)$, $L\in \{L_u,L_s\}$, be the classes 
of pseudodifferential operators with symbols in $S^{\comp}_{L,\rho,\rho'}$, $S^{\comp}_{L,\rho}$ defined in \cite[Appendix]{measupp}, following the same construction as in~\cite[\S3]{hgap}. In particular the operators in these classes are pseudolocal and bounded on $L^2(M)$ uniformly in $h$. However, the remainders will become $\mathcal{O}(h^{1-\rho-\rho'})$ or $\mathcal O(h^{1-\rho-})$ because of the assumptions on derivatives~\eqref{e:symbol-derby}.

We also have the following (non-canonical) quantization procedures
$$
\Op_h^L:a\in S^{\comp}_{L,\rho,\rho'}(T^*M\setminus 0)\ \mapsto\ \Op_h^L(a)\in \Psi^{\comp}_{h,L,\rho,\rho'}(M),
$$
and
$$
\Op_h^L:a\in S^{\comp}_{L,\rho}(T^*M\setminus 0)\ \mapsto\ \Op_h^L(a)\in \Psi^{\comp}_{h,L,\rho}(M).
$$

The $\Psi^{\comp}_{h,L,\rho}$ calculus satisfies a version of Egorov's Theorem
with logarithmically long time: For $A=\Op_h(a)$
where $a\in C_0^\infty(\{1/4<|\xi|_g<4\})$ is independent of $h$,
\begin{align}
	\label{e:long-egorov-1}
A(t)&=\Op_h^{L_s}(a\circ\varphi_t)+\mathcal O(h^{1-\rho-})_{L^2\to L^2},\\
	\label{e:long-egorov-2}
A(-t)&=\Op_h^{L_u}(a\circ\varphi_{-t})+\mathcal O(h^{1-\rho-})_{L^2\to L^2}
\end{align}
uniformly in $t\in [0,\rho\log(1/h)]$, see \cite[Appendix]{measupp}. 

Finally, we have the following norm bound similar to \eqref{e:basic-norm-bound}, which is a consequence of sharp G\aa rding inequality (applying to $(\sup|a|)I-\Op_h^L(a)\Op_h^L(a)^\ast$) in this calculus (see \cite[Appendix]{measupp})
\begin{equation}
	\label{e:exotic-norm-bound}
a\in S_{L,\rho,\rho'}^{\comp}
\quad\Longrightarrow\quad
\|\Op_h^L(a)\|_{L^2\to L^2}\leq \sup|a|+Ch^{1-\rho-\rho'}.
\end{equation}


\subsection{Semiclassical eigenvalue problem and damped propagator}
\label{s:semi-eigen}
We consider the following general semiclassical eigenvalue (or more precisely, quasimode) problem 
\begin{equation}
	\label{e:semi-eigen}
(\mathcal{P}(z,h)-z)u:=(P-ihQ(z)-z)u=\mathcal{O}(h^\infty), \quad \|u\|_{L^2(M)}=1.
\end{equation}
where
\begin{equation}
	\label{e:domain-z}
z=1+h\omega, \quad \omega=\mathcal{O}(1)\in\mathbb{C}
\end{equation}
Here the principal part $P$ is defined as in \eqref{e:the-P} and $Q=Q(z)$ is a family of operators holomorphically depending on the parameter $z$ (or equivalently $\omega$ in a fixed neighborhood of 0 in $\mathbb{C}$). 
Moreover $Q=Q(z)$ satisfies
\begin{equation}
	\label{e:the-Q}
Q(z)\in\Psi_h^{\comp},\quad \WF_h(Q(z))\subset p^{-1}(1/4,4)
\end{equation}
and its principal symbol
\begin{equation}
	\label{e:symbol-Q}
q:=\sigma_h(Q(z))\geq0
\end{equation} 
is independent of $z$ and $q\not\equiv0$ on $S^\ast M$. We shall also assume that
\begin{equation}
	\label{e:wf-u}
\WF_h(u)\subset S^\ast M=p^{-1}(1).
\end{equation}

Now we turn to the Schr\"{o}dinger equation $ih\partial_t\Psi=\mathcal{P}(z,h)\Psi$ associated to the eigenvalue problem \eqref{e:semi-eigen} and denote by 
\begin{equation}
	\label{e:U-q-t}
U_q(t):=\exp(-it\mathcal{P}(z,h)/h), \quad t\geq0
\end{equation}
the solution operator. We call $U_q(t)$ the \emph{damped propagator}.

For $t$ bounded uniformly in $h$, we consider the following operators
\begin{equation}
	\label{e:v+-}
V_-(t)=U_0(-t)U_q(t),  \quad \text{ and } \quad V_+(t)=U_q(t)U_0(-t)
\end{equation}
Since
\begin{equation*}
\frac{d}{dt}U_q(t)=-\frac{i}{h}(P-ihQ)U_q(t), \quad \text{ and } \quad
\frac{d}{dt}U_0(-t)=\frac{i}{h}U_0(-t)P,
\end{equation*}
by the product rule, we have
\begin{equation*}
\frac{d}{dt}V_-(t)=-U_0(-t)QU_q(t)=-Q(t)V_-(t)
\end{equation*}
where for $Q$ we use the notation \eqref{e:A-t}. By Egorov's theorem, we have $Q(t)\in\Psi^{\comp}_h$ uniformly in $t$ with symbol $q(t)=q\circ\varphi_t$ and thus we get $V_-(t)\in\Psi^0_h$ uniformly in $t$ with symbol (see \cite[\S 8.2]{e-z})
\begin{equation}
	\label{e:v-symbol}
v_-(t)=\exp\left(-\int_0^tq\circ\varphi_sds\right).
\end{equation}
Similarly, (or using $V_+(t)=U_0(t)V_-(t)U_0(-t)$), we see $V_+(t)\in\Psi^0_h$ with symbol
\begin{equation*}
v_+(t)=\exp\left(-\int_0^tq\circ\varphi_{-s}ds\right).
\end{equation*}

\begin{lem}
For any $t\geq0$ fixed,
\begin{equation}
	\label{e:bound-damp-1}
\|U_q(t)\|_{L^2(M)\to L^2(M)}\leq 1+Ch.
\end{equation}
As a corollary, if $0\leq t\leq T_0:=2\log(1/h)$, then
\begin{equation}
	\label{e:bound-damp}
\|U_q(t)\|_{L^2(M)\to L^2(M)}\leq 1+Ch^{1/2}.
\end{equation}
\end{lem}
\begin{proof}
Since $q\geq0$, we see $0<v_\pm(t)\leq1$ and thus by \eqref{e:basic-norm-bound},
\begin{equation*}
\|V_\pm(t)\|_{L^2(M)\to L^2(M)}\leq 1+Ch.
\end{equation*}
This finishes the proof of \eqref{e:bound-damp-1} since $U_0(-t)$ is unitary. To see \eqref{e:bound-damp}, we only need to let $t=t_0k$ with $t_0$ uniformly bounded in $h$ and $k\in\mathbb{N}$ such that $k\leq C\log(1/h)$. Using \eqref{e:bound-damp-1} for $U_q(t_0)$, we have
\begin{equation*}
\|U_q(t)\|_{L^2(M)\to L^2(M)}\leq \|U_q(t_0)\|^k_{L^2(M)\to L^2(M)}
\leq (1+Ch)^{C\log(1/h)}\leq 1+Ch^{1/2}.
\end{equation*}
\end{proof}


\subsection{Decay of the damped propagator and spectral gap}
\label{s:gap-semi}

In the next section, we prove the following result on the decay of the damped propagator after very long time. The choice of the time here corresponds to twice the Ehrenfest time. To state the theorem, we first introduce a microlocal cutoff operator $\Pi:=\chi(-h^2\Delta)\in\Psi_h^{\comp}$ defined by functional calculus with 
\begin{equation}
	\label{e:chi}
\chi\in C_0^\infty(\mathbb{R};[0,1]),\quad \supp\chi\subset((1-\delta)^2,(1+\delta)^2),
\quad \chi\equiv1 \text{ near } 1.
\end{equation}
Here $\delta\in(0,1/4)$ is small enough, chosen depending on $q$ later in \eqref{e:delta}.

\begin{thm}
	\label{t:propagator-decay}
Let $U_q(t)$ be defined as \eqref{e:U-q-t} and $\Pi$ be defined as above and let $T_0:=2\log(1/h)$. Then there exists $\beta_0>0$ depending only on $M$ and $q$, such that for all $0<h<1$, 
\begin{equation}
	\label{e:propagator-decay}
\|U_q(T_0)\Pi\|_{L^2(M)\to L^2(M)}\leq Ch^{\beta_0}.
\end{equation}
\end{thm}

As a corollary, we have the following result on the spectral gap for the semiclassical eigenvalue problem \eqref{e:semi-eigen}.

\begin{cor}
	\label{p:gap-semi}
There exists $\beta>0$ depending only on $M$ and $q$ and $h_0>0$, such that for $0<h<h_0$, if $z$ and $u$ satisfies the semiclassical eigenvalue problem \eqref{e:semi-eigen} and \eqref{e:domain-z} with the wavefront set condition \eqref{e:wf-u}, then
\begin{equation}
	\label{e:gap-semi}
\Im z\leq -\beta h.
\end{equation}
\end{cor}


\begin{proof}
By definition \eqref{e:U-q-t} of $U_q(t)$, we have
\begin{equation*}
\frac{d}{dt}e^{itz/h}U_q(t)=-\frac{i}{h}e^{itz/h}U_q(t)(\mathcal{P}(z,h)-z).
\end{equation*}
Applying this to $u$ and integrate from 0 to $T_0$, we get by \eqref{e:semi-eigen} and \eqref{e:bound-damp}
\begin{equation*}
e^{iT_0z/h}U_q(T_0)u=u+\mathcal{O}_{L^2}(h^\infty).
\end{equation*}
Here we also use that for $|t|\leq T_0=2\log(1/h)$ and $z$ satisfying \eqref{e:domain-z},
\begin{equation}
	\label{e:log-polynomial}
|e^{itz/h}|=e^{-t\Im z/h}\leq h^{-2|\Im z|/h}\leq h^{-N}
\end{equation}
for some $N>0$. By \eqref{e:wf-u}, we have $\Pi u=u+\mathcal{O}_{L^2}(h^\infty)$ and thus by \eqref{e:log-polynomial} again, 
\begin{equation*}
U_q(T_0)\Pi u=e^{-iT_0z/h}u+\mathcal{O}_{L^2}(h^\infty).
\end{equation*}
Taking the $L^2$-norm and using \eqref{e:propagator-decay}, we have $e^{T_0\Im z/h}\leq Ch^{\beta_0}$, and thus
\begin{equation*}
\frac{\Im z}{h}\leq\frac{1}{T_0}(\log C+\beta_0\log h)
=\frac{\log C}{2\log(1/h)}-\frac{\beta_0}{2},
\end{equation*}
which gives \eqref{e:gap-semi} with $\beta=\beta_0/4$ if $h$ is small enough.
\end{proof}

\section{Decay of the damped propagator}
\label{s:proof-propagator}
In this section, we prove Theorem \ref{t:propagator-decay} using a similar strategy as in \cite{measupp}.

\subsection{Partition of unity}
\label{s:partition}
We first recall the partition of unity used in \cite{measupp} with some modification adapting to our situation. First, we fix conic open sets 
$$
\mathcal U_1,\mathcal U_2\subset T^*M\setminus 0,\quad
\mathcal U_1,\mathcal U_2\neq \emptyset,\quad
\overline{\mathcal U_1}\cap\overline{\mathcal U_2}=\emptyset
$$
and we require that there exists $\delta=\delta(q)>0$ such that 
\begin{equation}
	\label{e:delta}
\min\left\{\left.\int_0^1 q\circ\varphi_s(x,\xi)ds\right|(x,\xi)\in\overline{\mathcal U_2}\cap p^{-1}(1-2\delta,1+2\delta)\right\}>0
\end{equation}
which is possible since $q\geq0$ and $q|_{S^\ast M}\not\equiv0$.

We introduce a pseudodifferential partition of unity
$$
I=A_0+A_1+A_2,\quad
A_0\in\Psi^0_h(M),\quad
A_1,A_2\in\Psi^{\comp}_h(M)
$$
such that
\begin{itemize}
\item $A_0$ is microlocalized away from the cosphere bundle $S^*M$.
More specifically, we put $A_0:=\psi_0(-h^2\Delta)$ where
$\psi_0\in C^\infty(\mathbb R;[0,1])$ satisfies
\begin{equation}
	\label{e:psi-0}
\supp\psi_0\cap [(1-\delta)^2,(1+\delta)^2]=\emptyset,\quad
\supp(1-\psi_0)\subset ((1-2\delta)^2, (1+2\delta)^2).
\end{equation}
This implies that
$$
\WF_h(A_0)\cap p^{-1}([1-\delta,1+\delta])=\emptyset,\quad
\WF_h(I-A_0)\subset p^{-1}(1-2\delta,1+2\delta).
$$
\item $A_1,A_2$ are microlocalized in an energy shell and away from $\mathcal U_1,\mathcal U_2$, that is
\begin{equation*}
\WF_h(A_1)\cup\WF_h(A_2)\subset p^{-1}(1-2\delta,1+2\delta),\quad
\WF_h(A_1)\cap \overline{\mathcal U_1}=\WF_h(A_2)\cap\overline{\mathcal U_2}=\emptyset.
\end{equation*}
\item $A_1$ is \emph{damped}, in the sense that there exists $\eta=\eta(q)>0$ such that
\begin{equation}
	\label{e:a-1-damp}
\supp a_1\subset\left\{(x,\xi):\int_0^1q\circ\varphi_sds>\eta\right\}.
\end{equation}
\item Finally, we choose $A_1,A_2$ so that
\begin{equation}
  \label{e:A-symbol}
0\leq a_\ell\leq 1\quad\text{where }
a_\ell:=\sigma_h(A_\ell),\quad \ell=0,1,2.
\end{equation}
\end{itemize}

For each $n\in\mathbb N_0$, we define the set of words of length $n$,
$$
\mathcal W(n):=\{1,2\}^n=\big\{\mathbf w=w_0\dots w_{n-1}\mid
w_0,\dots,w_{n-1}\in \{1,2\}\big\}.
$$
Recall that in \cite[\S 3.1]{measupp}, we dynamically refine the partition of unity and define the operators
\begin{equation}
	\label{e:op-word-0}
A_\mathbf{w}=A_{w_{n-1}}(n-1) A_{w_{n-2}}(n-2)\cdots A_{w_1}(1) A_{w_0}(0)
\end{equation}
with symbols 
\begin{equation}
	\label{e:symbol-word-0}
a_{\mathbf{w}}=\prod_{j=0}^{n-1} \big(a_{w_j}\circ\varphi_j\big).
\end{equation}

Instead, we define for each word $\mathbf{w}\in\mathcal{W}(n)$, the damped propagator corresponding to the word $\mathbf{w}$ as
\begin{equation}
	\label{e:damp-word}
U_\mathbf{w}:=U_q(1)A_{w_{n-1}}U_q(1)A_{w_{n-2}}\cdots U_q(1)A_{w_0}.
\end{equation}
We also define the following operators which are the ``damped'' analogues of $A_{\mathbf{w}}$:
\begin{equation}
	\label{e:op-word}
A_\mathbf{w}^-=U_0(-n)U_\mathbf{w},\quad
A_{\mathbf{w}}^+=U_{\mathbf{w}}U_0(-n)=A_{\mathbf{w}}^-(-n).
\end{equation}

If $n$ is bounded independently of $h$, then $A_\mathbf{w}^\pm\in \Psi^{\comp}_h(M)$ with principal symbol $\sigma_h(A_{\mathbf w}^\pm)=a_{\mathbf w}^\pm$ given by
\begin{equation}
	\label{e:symbol-word}
\begin{split}
a_{\mathbf w}^-=\prod_{j=0}^{n-1}\left[a_{w_j}\exp\left(-\int_0^1q\circ\varphi_sds\right)\right]\circ\varphi_j,\\
a_{\mathbf w}^+=a_{\mathbf{w}}^-\circ\varphi_{-n}=\prod_{j=1}^n\left[a_{w_{n-j}}\exp\left(-\int_0^1q\circ\varphi_sds\right)\right]\circ\varphi_{-j}.
\end{split}
\end{equation}
To see this, we notice that $A_{\mathbf{w}}$ defined in \eqref{e:op-word-0} can be rewritten as
\begin{equation*}
A_{\mathbf{w}}=U_0(-n)U_0(1)A_{w_{n-1}}U_0(1)A_{w_{n-2}}\cdots U_0(1)A_{w_0}.
\end{equation*}
Therefore $A_{\mathbf{w}}^-$ is exactly $A_{\mathbf{w}}$ with all $U_0(1)$ replaced by $U_q(1)$ or equivalently with $A_j$ replaced by 
\begin{equation}
\label{e:damp-op-A}
\widetilde{A}_j:=U_0(-1)U_q(1)A_j=V_-(1)A_j,\quad j=1,2.
\end{equation}
Meanwhile, $a_{\mathbf{w}}^-$ is exactly $a_{\mathbf{w}}$ defined in \eqref{e:symbol-word-0} with $a_j$ replaced by the symbol of $\widetilde{A}_j$ which is given by (see \eqref{e:v-symbol})
\begin{equation}
	\label{e:damp-a}
\widetilde{a}_j:=a_jv_-(1)=a_j\exp\left(-\int_0^1q\circ\varphi_sds\right).
\end{equation}
This allows us to apply the theory developed in \cite{measupp} for $A_{\mathbf{w}}$ to $A_{\mathbf{w}}^\pm$. By \eqref{e:a-1-damp}, we have
\begin{equation}
	\label{e:damp-a-a}
0\leq\widetilde{a}_1\leq e^{-\eta}a_1,\quad 0\leq \widetilde{a}_2\leq a_2. 
\end{equation}

For a subset $\mathcal E\subset \mathcal W(n)$, we also define
the operators $U_{\mathcal{E}}$, $A_{\mathcal E}^\pm$ and the symbol $a_{\mathcal E}^\pm$ by
\begin{equation}
	\label{e:A-subset}
U_{\mathcal E}:=\sum_{\mathbf w\in \mathcal E}U_{\mathbf w},\quad
A^\pm_{\mathcal E}:=\sum_{\mathbf w\in \mathcal E}A^\pm_{\mathbf w},
\quad
a^\pm_{\mathcal E}:=\sum_{\mathbf w\in \mathcal E}a^\pm_{\mathbf w}.
\end{equation}

In particular, we have 
\begin{equation}
	\label{e:full-word}
U_{\mathcal{W}(n)}=\left(U_q(1)(A_1+A_2)\right)^n=\left(U_q(1)(I-A_0)\right)^n.
\end{equation}
Moreover since $U_0(-n)$ is unitary, 
\begin{equation}
	\label{e:equal-norm}
\|U_{\mathcal{E}}\|_{L^2\to L^2}=\|A_{\mathcal E}^\pm\|_{L^2\to L^2}.
\end{equation}

\subsection{Long words and damped propagation}
Now we can proceed as in \cite{measupp}. Take $\rho\in (0,1)$ very close to 1, to be chosen later (in \eqref{e:fup}), and put
$$
N_0:=\Big\lceil\frac{\rho}{4}\log(1/h)\Big\rceil\in\mathbb N,\quad
N_1:=4N_0\approx \rho\log(1/h).
$$
Then we have the following lemma parallel to \cite[Lemma 3.2, Lemma 4.4]{measupp} with the same statements for $a_{\mathbf{w}}$ and $A_{\mathbf{w}}$.
\begin{lem}
	\label{l:long-word-egorov}
For each $\mathbf w\in \mathcal W(N_0)$ we have (with bounds independent of $\mathbf w$)
\begin{equation}
	\label{e:lwe-1}
a_{\mathbf w}^-\in S^{\comp}_{L_s,\rho/4}(T^*M\setminus 0),\quad
A_{\mathbf w}^-=\Op_h^{L_s}(a_{\mathbf w}^-)+\mathcal O(h^{3/4})_{L^2\to L^2}.
\end{equation}
If instead $\mathbf w\in \mathcal W(N_1)$, then
\begin{equation}
	\label{e:lwe-2}
a_{\mathbf w}^-\in S^{\comp}_{L_s,\rho}(T^*M\setminus 0),\quad
A_{\mathbf w}^-=\Op_h^{L_s}(a_{\mathbf w}^-)+\mathcal O(h^{1-\rho-})_{L^2\to L^2}.
\end{equation}
Moreover, for any subset $\mathcal{E}$ of $\mathcal{W}(N_0)$, we have
\begin{equation}
	\label{e:lwe-3}
a_{\mathcal E}^-\in S^{\comp}_{L_s,1/2,1/4}(T^*M\setminus 0),\quad
A_{\mathcal E}^-=\Op_h^{L_s}(a_{\mathcal E}^-)+\mathcal O(h^{1/2})_{L^2\to L^2}.
\end{equation}
All of above are true if we replace the sign $-$ by $+$ and $L_s$ by $L_u$.
\end{lem}
\begin{proof}
The proof of the statements for $a_{\mathbf{w}}^-$ and $A_{\mathbf{w}}^-$ is essentially the same as \cite[Lemma 3.2, Lemma 4.4]{measupp} with $a_j$ replaced by $\widetilde{a}_j$ defined in \eqref{e:damp-a} and $A_j$ replaced by $\widetilde{A}_j$ defined in \eqref{e:damp-op-A}. We refer the details of the proof to \cite{measupp}. For $a_{\mathbf{w}}^+$ and $A_{\mathbf{w}}^+$ we simply need to reverse the direction of the flow $\varphi_t$ and notice that this exchanges the unstable foliation $L_u$ and stable foliation $L_s$.
\end{proof}

Now, we define the set of \emph{damped} words. We fix a parameter $\alpha\in(0,1)$ and define the set of damped words of length $N_0$ to be
\begin{equation}
	\label{e:Z-damp-set}
\mathcal Z:=\{\mathbf{w}=w_0w_1\cdots w_{N_0-1}\in\mathcal{W}(N_0)\mid
\#\{j\in\{0,\ldots,N_0-1\}\mid w_j=1\}\geq \alpha N_0\}.
\end{equation}
Next we define the set of damped words $\mathcal Y\subset \mathcal W(2N_1)$ by iterating~$\mathcal Z$. More specifically, we write words in $\mathcal W(2N_1)$ as concatenations $\mathbf w^{(1)}\dots \mathbf w^{(8)}$ where each of the words $\mathbf w^{(1)},\dots,\mathbf w^{(8)}\in \mathcal W(N_0)$, define the partition
\begin{equation}
	\label{e:XY-def}
\begin{aligned}
\mathcal W(2N_1)&=\mathcal X\sqcup\mathcal Y,\\
\mathcal X&:=\{\mathbf w^{(1)}\dots \mathbf w^{(8)}\mid \mathbf w^{(\ell)}\notin\mathcal Z\quad\text{for all }\ell\},\\
\mathcal Y&:=\{\mathbf w^{(1)}\dots \mathbf w^{(8)}\mid \text{there exists $\ell$ such that }\mathbf w^{(\ell)}\in\mathcal Z\}
\end{aligned}
\end{equation}

In our argument the parameter $\alpha$ will be taken small (in Proposition \ref{p:bound-long-words}) so that the set $\mathcal X$ is not too large. The size of $\mathcal X$ is estimated by the following lemma (see \cite[Lemma 3.3]{measupp})
\begin{lem}
	\label{l:X-count}
The number of elements in $\mathcal X$ is bounded by (here $C$ may depend on $\alpha$)
\begin{equation}
	\label{e:X-count}
\#(\mathcal X)\leq Ch^{-4\sqrt{\alpha}}.
\end{equation}
\end{lem}

\subsection{Estimating damped words}
\label{s:damp}
We start by proving the following norm bound on the set of damped words.
\begin{prop}
	\label{p:u-y}
There exist $C,\beta_{\mathrm{damp}}>0$ only depending on $\alpha$ and $q$ such that
\begin{equation}
	\label{e:norm-damped}
\|U_{\mathcal{Y}}\|_{L^2\to L^2}\leq Ch^{\beta_{\mathrm{damp}}}.
\end{equation}
\end{prop}
\begin{proof}
First, if we define $\mathcal{Q}:=\mathcal{W}(N_0)\setminus\mathcal{Z}$, since
$$
\mathcal Y=\bigsqcup_{\ell=1}^{8}\mathcal Y_\ell,\quad
\mathcal Y_\ell := \{\mathbf w^{(1)}\dots \mathbf w^{(8)}\mid
\mathbf w^{(\ell)}\in \mathcal Z,\quad
\mathbf w^{(\ell+1)},\dots,\mathbf w^{(8)}\in\mathcal Q\}.
$$
 we can rewrite $U_{\mathcal{Y}}$ as 
\begin{equation*}
U_{\mathcal{Y}}=\sum_{\ell=1}^8U_{\mathcal{Q}}^{8-\ell}U_{\mathcal{Z}}U_{\mathcal{W}(N_0)}^{\ell-1}.
\end{equation*}
From \eqref{e:equal-norm}, \eqref{e:exotic-norm-bound} and \eqref{e:lwe-3}, we see that for any subset $\mathcal{E}$ of $\mathcal{W}(N_0)$,
\begin{equation*}
\|U_{\mathcal{E}}\|_{L^2\to L^2}
=\|A_{\mathcal{E}}^-\|_{L^2\to L^2}
\leq \|\Op_h^{L_s}(a_{\mathcal{E}}^-)\|_{L^2\to L^2}+Ch^{1/2}
\leq \sup|a_{\mathcal{E}}^-|+Ch^{1/4}
\end{equation*}
From the definition \eqref{e:A-subset} of $a_{\mathcal{E}}^-$ and \eqref{e:damp-a-a}, we see that at any point $(x,\xi)$, by \eqref{e:A-symbol},
$$|a_{\mathcal{E}}^-|\leq|a_{\mathcal{W}(N_0)}^-|\leq
\sum_{\mathbf{w}\in\mathcal{W}(N_0)}\prod_{j=0}^{N_0-1}\widetilde{a}_{w_j}\circ\varphi_j=\prod_{j=0}^{N_0-1}(\widetilde{a}_1+\widetilde{a}_2)\circ\varphi_j\leq 1.$$
Therefore we have
\begin{equation*}
\|U_{\mathcal{Q}}\|_{L^2\to L^2}\leq 1+Ch^{1/4},
\quad
\|U_{\mathcal{W}(N_0)}\|_{L^2\to L^2}\leq 1+Ch^{1/4}
\end{equation*}
Finally, in the definition \eqref{e:A-subset} for $a_{\mathcal{Z}}^-$,
\begin{equation*}
a_{\mathcal{Z}}^-=\sum_{\mathbf{w}\in\mathcal{Z}}a_{\mathbf{w}}^-
=\sum_{\mathbf{w}\in\mathcal{Z}}\prod_{j=0}^{N_0-1}\tilde{a}_{w_j}\circ\varphi_j
\end{equation*}
for each $\mathbf{w}\in\mathcal{Z}$, by the definition \eqref{e:Z-damp-set}
of $\mathcal{Z}$, there are at least $\alpha N_0$ letters $w_j$ are 1. Therefore by \eqref{e:damp-a-a} and \eqref{e:A-symbol},
\begin{equation*}
\begin{split}
|a_{\mathcal{Z}}^-|
\leq e^{-\eta\alpha N_0}\sum_{\mathbf{w}\in\mathcal{Z}}\prod_{j=0}^{N_0-1}a_{w_j}\circ\varphi_j\leq &\; e^{-\eta\alpha N_0}\sum_{\mathbf{w}\in\mathcal{W}(N_0)}\prod_{j=0}^{N_0-1}a_{w_j}\circ\varphi_j\\
=&\; e^{-\eta\alpha N_0}\prod_{j=0}^{N_0-1}(a_1+a_2)\circ\varphi_j\leq e^{-\eta\alpha N_0}\leq Ch^{\alpha\rho\eta/4}.
\end{split}
\end{equation*}
Therefore we have $$\|U_\mathcal{Z}\|_{L^2\to L^2}\leq Ch^{\beta_{\mathrm{damp}}}$$ with 
$\beta_{\mathrm{damp}}=\min(\alpha\rho\eta/4,1/4)$ and this finishes the proof.
\end{proof}

\subsection{Fractal uncertainty principle}
\label{s:fup}
For each word $\mathbf{w}\in\mathcal{W}(2N_1)$, we have a uniform bound on the norm for every $U_{\mathbf{w}}$ which is a consequence of the fractal uncertainty principle \cite[Proposition 5.7]{measupp}.

\begin{prop} 
	\label{p:fup}
There exists $C,\beta_{\mathrm{FUP}}>0$ depending only on $M$, $\mathcal{U}_1$ and $\mathcal{U}_2$ such that for all $\mathbf{w}\in\mathcal{W}(2N_1)$,
\begin{equation}
	\label{e:fup-words}
\|U_\mathbf{w}\|_{L^2\to L^2}\leq Ch^{\beta_{\mathrm{FUP}}}.
\end{equation}
\end{prop}

\begin{proof}
Write $\mathbf{w}=\mathbf{w}_-\mathbf{w}_+$ with $\mathbf{w}_\pm\in\mathcal{W}(N_1)$, then we have
\begin{equation*}
U_0(-N_1)U_{\mathbf{w}}U_0(-N_1)=A_-A_+,
\quad A_\pm:=A_{\mathbf{w}_\pm}^\pm.
\end{equation*}
By Lemma \ref{l:long-word-egorov}, we see
\begin{equation}
	\label{e:a-a+}
A_-=\Op_h^{L_s}(a_-)+\mathcal{O}(h^{1-\rho-})_{L^2\to L^2},
\quad A_+=\Op_h^{L_u}(a_+)+\mathcal{O}(h^{1-\rho-})_{L^2\to L^2}
\end{equation}
where $a_\pm:=a_{\mathbf{w}_\pm}^\pm$. The same argument as in \cite[Section 5]{measupp} shows that $\supp a_\pm$ are porous sets and by \cite[Proposition 5.7]{measupp}, we obtain for some $\rho\in(0,1)$,
\begin{equation}
	\label{e:fup}
\|\Op_h^{L_s}(a_-)\Op_h^{L_u}(a_+)\|_{L^2\to L^2}\leq Ch^{\beta_{\mathrm{FUP}}}
\end{equation}
where $\rho$ and $\beta_{\mathrm{FUP}}>0$ depending only on $M,\mathcal{U}_1$ and $\mathcal{U}_2$. Combining \eqref{e:a-a+} and \eqref{e:fup}, we obtain \eqref{e:fup-words}.
\end{proof}

\subsection{End of the proof of Theorem \ref{t:propagator-decay}}
\label{s:end-proof}
Combining Proposition \ref{p:u-y}, \ref{p:fup} and Lemma \ref{l:X-count} and writing $U_{\mathcal{W}(2N_1)}=U_{\mathcal{X}}+U_{\mathcal{Y}}$, we get the following bound on the norm of damped words provided that we choose $\alpha$ small, say $\alpha=\beta_{\mathrm{FUP}}^2/64$ with $\beta_{\mathrm{FUP}}$ in \eqref{e:fup-words}, so that $\beta_{\mathrm{FUP}}-4\sqrt{\alpha}>0$.

\begin{prop}
	\label{p:bound-long-words}
There exists $C,\beta>0$ depending only on $M$ and $q$ such that
\begin{equation}
	\label{e:bound-long-words}
\|U_{\mathcal{W}(2N_1)}\|_{L^2\to L^2}\leq Ch^\beta.
\end{equation}
\end{prop}

Now to prove Theorem \ref{t:propagator-decay}, we only need to estimate the difference between the operators $U_{\mathcal{W}(2N_1)}\Pi$ and $U_q(2N_1)\Pi$. We first expand $U_q(2N_1)$ as 
\begin{equation*}
\begin{split}
U_q(2N_1)=U_q(1)^{2N_1}=&[U_q(1)(I-A_0)+U_q(1)A_0]^{2N_1}\\
=&U_{\mathcal{W}(2N_1)}+\sum_{\ell=1}^{2N_1}
U_q(1)^{2N_1-\ell}U_q(1)A_0(U_q(1)(I-A_0))^{\ell-1}.
\end{split}
\end{equation*}
Recalling that $\Pi=\chi(-h^2\Delta)$ and $A_0=\psi_0(-h^2\Delta)$ with $\chi, \psi_0$ as in \eqref{e:chi} and \eqref{e:psi-0}, respectively, we have 
$$\Pi A_0=A_0\Pi=0, \quad (I-A_0)\Pi=\Pi$$
and thus
\begin{equation*}
\begin{split}
U_q(2N_1)\Pi-U_{\mathcal{W}(2N_1)}\Pi
=&\sum_{\ell=1}^{2N_1}U_q(1)^{2N_1-\ell+1}A_0(U_q(1)(I-A_0))^{\ell-1}\Pi\\
=&\sum_{\ell=1}^{2N_1}U_q(1)^{2N_1-\ell+1}A_0[(U_q(1)(I-A_0))^{\ell-1},\Pi]
\end{split}
\end{equation*}
Here the commutator can be written as
\begin{equation*}
[(U_q(1)(I-A_0))^{\ell-1},\Pi]=
\sum_{j=1}^{\ell-1}(U_q(1)(I-A_0))^{\ell-1-j}[U_q(1)(I-A_0),\Pi](U_q(1)(I-A_0))^{j-1}.
\end{equation*}
Recalling the definition \eqref{e:v+-} of $V_-$, and in particular, $V_-(1)\in\Psi^0_h$, we have
$[V_-(1),\Pi]=\mathcal{O}_{L^2\to L^2}(h)$ and thus
\begin{equation*}
[U_q(1)(I-A_0),\Pi]=[U_0(1)V_-(1)(I-A_0),\Pi]=U_0(1)[V_-(1),\Pi](I-A_0)=\mathcal{O}_{L^2\to L^2}(h)
\end{equation*}
where we also use $[U_0(1),\Pi]=[I-A_0,\Pi]=0$. Now by \eqref{e:bound-damp} and 
$2N_1\leq 2\log(1/h)$ 
\begin{equation*}
\|U_q(2N_1)\Pi-U_{\mathcal{W}(2N_1)}\Pi\|_{L^2\to L^2}\leq Ch^{1/2}.
\end{equation*}
This combining with \eqref{e:bound-long-words} shows that
\begin{equation*}
\|U_q(2N_1)\Pi\|_{L^2\to L^2}\leq Ch^\beta.
\end{equation*}
Recalling that $0\leq 2N_1=8\lceil\frac{\rho}{4}\log(1/h)\rceil\leq T_0=2\log(1/h)$,  we can write $U_q(T_0)=U_q(T_0-2N_1)U_q(2N_1)$ and applying \eqref{e:bound-damp} to $U_q(T_0-2N_1)$ to finish the proof of Theorem \ref{t:propagator-decay}.


\section{Proof of the theorems}
\label{s:proof}
Now we go back to the setting of the damped wave equation \eqref{e:dampwave} and the corresponding eigenvalue problem \eqref{e:eigen}.  

\subsection{Spectral gap}
\label{s:gap}

To prove Theorem \ref{t:gap}, we rescale and reduce the original eigenvalue problem \eqref{e:eigen} (with $\Re\tau\to+\infty$ and $\Im\tau=\mathcal{O}(1)$) to the semiclassical one \eqref{e:semi-eigen} with a suitable $q$. 

First, we write $\tau=h^{-1}+\omega$ where $\omega=\mathcal{O}(1)$ in $\mathbb{C}$, then we obtain
\begin{equation*}
h^2P(\tau)u=(-h^2\Delta-2ihza(x)-z^2)u=0, \quad \|u\|_{L^2}=1
\end{equation*}
with $z=h\tau=1+h\omega$ satisfying \eqref{e:domain-z}. Since $\sigma_h(-h^2\Delta)=p^2$, $u$ satisfies the wavefront set condition \eqref{e:wf-u} by standard elliptic estimates.
This allow us to work near $S^\ast M$ microlocally. More precisely, we fix  functions $\psi_1,\psi_2\in C_0^\infty((0,\infty);\mathbb{R})$ satisfying 
$$\supp\psi_1\subset(1/16,16), \quad \psi_1\geq0, \quad \psi_1\equiv1 \text{ on } [1/4,4]$$
and $\psi_2(\lambda)=\psi_P(\lambda)^2/\lambda$ where $\psi_P$ is given by \eqref{e:psi-P}. Then by functional calculus, 
$$\Pi_j:=\psi_j(-h^2\Delta)\in\Psi^{\comp}_h(M),\quad j=1,2$$
satisfy the following wavefront set condition
$$\WF_h(\Pi_1)\subset p^{-1}((1/4,4)),\quad
\WF_h(I-\Pi_1)\cap p^{-1}([1/2,2])=\emptyset,$$
and
$$\WF_h(I-\Pi_2)\cap p^{-1}([1/4,4])=\emptyset.$$
In particular, by \eqref{e:wf-u}, we have
\begin{equation}
	\label{e:u-localize}
\Pi_ju=u+\mathcal{O}(h^\infty), \quad j=1,2.
\end{equation}

By definition \eqref{e:the-P} of $P$, we have $P^2=-h^2\Delta\Pi_2$. Let $P_1=a(x)\Pi_1$, then \eqref{e:u-localize} implies 
$$(P^2-2ihzP_1-z^2)u=(-h^2\Delta\Pi_2-2ihza(x)\Pi_1-z^2)u
=\mathcal{O}(h^\infty).$$
Thus it suffices to write
\begin{equation}
	\label{e:square-root}
P^2-2ihzP_1=(P-ihQ(z))^2+\mathcal{O}(h^\infty)
\end{equation}
for some $Q=Q(z)$ satisfying \eqref{e:the-Q} and \eqref{e:symbol-Q}, then we have
\begin{equation*}
(P-ihQ(z)+z)(P-ihQ(z)-z)u=\mathcal{O}(h^\infty)
\end{equation*}
and thus \eqref{e:semi-eigen} by the ellipticity of $P-ihQ(z)+\omega$.

To get \eqref{e:square-root}, we use a construction similar to \cite[\S4.2]{hgap} and express $Q$ as an asymptotic sum 
\begin{equation}
	\label{e:borelsum}
Q(z)\sim Q_0+hQ_1(z)+h^2Q_2(z)+\cdots
\end{equation}
with each $Q_j\in\Psi^{\comp}_h, j=1,2,\cdots$ holomorphically depending on $z$ and $Q_0\in\Psi^{\comp}_h$ independent of $z$. First we pick $Q_0$ with symbol 
\begin{equation}
	\label{e:q0}
q_0:=\sigma_h(Q_0)=\sigma_h(P_1)/p=a(x)\psi_1(|\xi|_g^2)/p(x,\xi),
\end{equation}
 then we have
\begin{equation*}
P^2-2ihzP_1=(P-ihQ_0)^2+h^2R_0(z)+\mathcal{O}(h^\infty)
\end{equation*}
for some $R_0(z)\in\Psi^{\comp}_h$. Next we can choose $Q_1$ with symbol $q_1:=\sigma_h(Q_1)=i\sigma_h(R_0)/2p$, so that
\begin{equation*}
P^2-2ihzP_1=(P-ih(Q_0+hQ_1(z)))^2+h^3R_1(z)+\mathcal{O}(h^\infty)
\end{equation*}
for some $R_1(z)\in\Psi^{\comp}_h$. We can continue this process to get a sequence of operators $Q_j(z)\in\Psi_h^{\comp}(M)$ such that for any $m\in\mathbb{N}$, we have
\begin{equation*}
P^2-2ihzP_1=(P-ih(Q_0+hQ_1(z)+\cdots+h^mQ_m(z)))^2+h^{m+2}R_m(z)+\mathcal{O}(h^\infty)
\end{equation*}
for some $R_m(z)\in\Psi^{\comp}_h$. Moreover, from the construction, it is not hard to see that we can take all $Q_j(z)$ and $R_j(z)$ depending on $z$ holomorphically and satisfying
\begin{equation*}
\WF_h(Q_j(z)),\WF_h(R_j(z))\subset p^{-1}(1/4,4).
\end{equation*}
Therefore the asymptotic sum \eqref{e:borelsum} satisfy \eqref{e:the-Q}.
Finally $Q_0$ is independent of $z$ and it is easy to check that $q:=\sigma_h(Q)=q_0$ defined in \eqref{e:q0} satisfy \eqref{e:symbol-Q}. 

Now Theorem \ref{t:gap} follows from Corollary \ref{p:gap-semi} by rescaling.

\subsection{Resolvent estimates}
\label{s:resolvent}
We denote the resolvent operator of the eigenvalue problem \eqref{e:eigen} by
\begin{equation}
	\label{e:resolvent}
R(\tau):=P(\tau)^{-1}=(-\Delta-\tau^2-2ia\tau)^{-1}.
\end{equation}
It is related to the resolvent of the operator $\mathcal{B}$ defined in \eqref{e:B-matrix} by the following formula
\begin{equation*}
(\tau-\mathcal{B})^{-1}=\begin{pmatrix}
R(\tau)(-2ia-\tau) & -R(\tau)\\
R(\tau)(2ia\tau-\tau^2)-I & -\tau R(\tau)
\end{pmatrix}
\end{equation*}

We have the following theorem giving a polynomial resolvent bound in a strip near the real axis when $|\tau|$ is large. We refer to \cite{nozw2} for a similar estimate in the situation of chaotic scattering where we borrow the idea of the proof.

\begin{thm}
\label{t:resolvent}
There exists $C$, $C_0$ and $\beta>0$ such that for $|\Re\tau|\geq C_0$, $\Im\tau>-\beta$,
\begin{equation}
	\label{e:resolvent-0-0}
\|R(\tau)\|_{L^2\to L^2}\leq C|\tau|^{-1-2\min(0,\Im\tau)}\log|\tau|.
\end{equation}
\end{thm}


\begin{proof}
We write $R(z,h)=h^{-2}R(\tau)=(-h^2\Delta-z^2-2ihza(x))^{-1}$ where $z=h\tau=1+h\omega$ as in Section \ref{s:gap}. To estimate the norm of $R(z,h)$, we build an approximation as follows. Let $Q(z)$ be the operator \eqref{e:borelsum}
constructed in Section \ref{s:gap} and let $U_q(t)$ be the damped propagator defined as in \eqref{e:U-q-t}. We take
\begin{equation}
	\label{e:reso-appro-0}
R_0(z,h)=\frac{i}{h}\int_0^{T_0}e^{itz/h}U_q(t)\Pi dt,
\end{equation}
with $T_0=2\log(1/h)$ as in Theorem \ref{t:propagator-decay}. A straightforward computation shows that
\begin{equation*}
(P-ihQ(z)-z)R_0(z,h)=\Pi-e^{iT_0z/h}U_q(T_0)\Pi
\end{equation*}
where by \eqref{e:propagator-decay},
\begin{equation}
	\label{e:error-0}
\|e^{iT_0z/h}U_q(T_0)\Pi\|_{L^2\to L^2}\leq Ce^{-T_0(\Im z)/h}h^{\beta_0}
=Ch^{\beta_0+2(\Im z)/h}.
\end{equation}
Moreover, by \eqref{e:bound-damp},
\begin{equation}
	\label{e:reso-bound-0}
\begin{split}
\|R_0(z,h)\|_{L^2\to L^2}\leq &\;h^{-1}\int_0^{T_0}e^{-t(\Im z)/h}\|U_q(t)\|_{L^2\to L^2}dt\\
\leq &\;CT_0h^{-1}\max\{1,e^{-T(\Im z)/h}\}\leq Ch^{-1+2\min(0,\Im z)/h}\log(1/h).
\end{split}
\end{equation}
To get a global approximate resolvent, we can use a standard gluing argument. More precisely, let $\widetilde{\chi}\in C_0^\infty(\mathbb{R};[0,1])$ satisfy
\begin{equation*}
\supp\widetilde{\chi}\subset\{\chi=1\}, \quad \widetilde{\chi}=1 \text{ near } 1
\end{equation*}
and define $\widetilde{\Pi}=\widetilde{\chi}(-h^2\Delta)\in\Psi_h^{\comp}$, then since $P-ihQ(z)+z$ and $\Pi$ are elliptic on $\WF_h(\widetilde{\Pi})$, we can find $\widetilde{R_0}(z,h)\in\Psi^{\comp}_h$ such that
\begin{equation*}
(P-ihQ(z)+z)\Pi\widetilde{R_0}(z,h)=\widetilde{\Pi}+\mathcal{O}_{L^2\to L^2}(h^\infty).
\end{equation*}
Then we have
\begin{equation*}
(P^2-2ihzP_1-z^2)R_0(z,h)\widetilde{R}_0(z,h)=\widetilde{\Pi}+\mathcal{O}_{L^2\to L^2}(h^{\beta_0+2(\Im z)/h}).
\end{equation*}

We claim that for $0\leq t\leq T_0=2\log(1/h)$ and $j=1,2$,
\begin{equation}
\label{e:propagate-pi}
(I-\Pi_j)U_q(t)\Pi=\mathcal{O}_{L^2\to H^2}(h^\infty)
\end{equation}
and thus we have
\begin{equation*}
(-h^2\Delta-z^2-2ihza(x))R_0(z,h)\widetilde{R}_0(z,h)=\widetilde{\Pi}+\mathcal{O}_{L^2\to L^2}(h^{\beta_0+2(\Im z)/h}).
\end{equation*}

By the ellipticity of $-h^2\Delta-z^2-2ihza(x)$ away from $S^\ast M$, we can also find $R_1(z,h)\in\Psi^{-2}_h$ such that
\begin{equation*}
(-h^2\Delta-z^2-2ihza(x))R_1(z,h)=I-\widetilde{\Pi}+\mathcal{O}_{L^2\to L^2}(h^\infty).
\end{equation*}
Therefore if we let $\widetilde{R}=R_0\widetilde{R}_0+R_1$, then
\begin{equation}
\label{e:appro-resolvent}
(-h^2\Delta-z^2-2ihza(x))\widetilde{R}(z,h)=I+\mathcal{O}_{L^2\to L^2}(h^{\beta_0+2(\Im z)/h})
\end{equation}
and by \eqref{e:reso-bound-0}, we have the following norm bound on the approximate resolvent
\begin{equation}
\label{e:reso-appro-bound}
\|\widetilde{R}(z,h)\|_{L^2\to L^2}\leq Ch^{-1+2\min(0,\Im z)/h}\log(1/h).
\end{equation}
Multiplying $R(z,h)$ to the left of both sides of \eqref{e:appro-resolvent}, we get 
$$R(z,h)-\widetilde{R}(z,h)=\mathcal{O}_{L^2\to L^2}(h^{\beta_0+2(\Im z)/h})$$
and as long as $\Im z>-\beta_0/4$, we see that for any $\varepsilon>0$,
\begin{equation}
\label{e:reso-bound}
\|R(z,h)\|_{L^2\to L^2}\leq Ch^{-1+2\min(0,\Im z)/h}\log(1/h).
\end{equation}
Rescaling \eqref{e:reso-bound} back to $\tau=h^{-1}z$ we get \eqref{e:resolvent-0-0}. Now to finish the proof, it remains to show \eqref{e:propagate-pi}. For $t$ uniformly bounded in $h$, \eqref{e:propagate-pi} follows from Egorov's Theorem since $\varphi_t$ leaves each energy surface $p^{-1}(E)=\{|\xi|=E\}$ invariant and by definition of $\Pi$ and $\Pi_j$,
\begin{equation*}
\WF_h(\Pi)\subset p^{-1}((1-\delta,1+\delta)), 
\quad\WF_h(I-\Pi_j)\subset T^\ast M\setminus p^{-1}((1/2,2)).
\end{equation*}
This can be extended to $0\leq t\leq 2\log(1/h)$ by the same argument as in \cite{anno,nozw,nozw2}. For example, we can first extend to $0\leq t\leq\frac{1}{4}\log(1/h)$ using Egorov's Theorem up to Ehrenfest time, then insert finite number of intermediate microlocal cutoffs in the middle to extend the estimate \eqref{e:propagate-pi} all the way to $0\leq t\leq 2\log(1/h)$.
\end{proof}

\subsection{From resolvent estimate to energy decay}
\label{s:energydecay}
Theorem \ref{t:energy-decay} follows from the resolvent estimate (Theorem \ref{t:resolvent}) by a standard argument. For example, we can apply \cite[Theorem 6.1]{csvw} which is the general statement for obtaining energy decay of damped wave equation from high frequency resolvent estimate. Here we give a sketch of the proof of Theorem \ref{t:energy-decay} following \cite[Chapter 5]{e-z} and  \cite[Section 3]{schenckpressure} where we refer the reader to a detailed argument.

We first remark that \eqref{e:resolvent-0-0} implies that for $|\Re\tau|\geq C_0$ and $\Im\tau>-\beta$,
\begin{equation}
\label{e:resolvent-0-2}
\|R(\tau)\|_{L^2\to H^2}\leq C|\tau|^{1-2\min(0,\Im\tau)}\log|\tau|.
\end{equation}
An interpolation between \eqref{e:resolvent-0-0} and \eqref{e:resolvent-0-2} shows that as long as $s>-2\min(0,\Im\tau)$, we have for $|\Re\tau|\geq C_0$ and $\Im\tau>-\beta$,
\begin{equation}
\label{e:resolvent-0-s}
\|R(\tau)\|_{L^2\to H^{1-s}}\leq C.
\end{equation}
Using standard arguments by induction, we can prove the \eqref{e:resolvent-0-s} for the $H^{s'}\to H^{1-s+s'}$ norm for any integer $s'$ and then by interpolation, any $s'\in\mathbb{R}$. In particular, with $s'=s$, we see that for any $s>0$, we have with $\gamma\in(0,\min(G,s/2))$, where $G$ is the spectral gap defined in \eqref{e:gap},
\begin{equation}
\label{e:resolvent-s-1}
\sup_{\Im\tau=-\gamma}\|R(\tau)\|_{H^s\to H^1}\leq C.
\end{equation}

Let $\chi\in C^\infty(\mathbb{R};[0,1])$ such that $\chi\equiv0$ for $t\leq0$ and $\chi\equiv1$ for $t\geq1$. With $v$ the solution of \eqref{e:dampwave}, we let $w=\chi v$, then
\begin{equation*}
(\partial_t^2-\Delta+2a(x)\partial_t)w=g:=\chi''v+2\chi'\partial_tv+2a\chi'v.
\end{equation*}
Taking inverse Fourier transform $\check{f}(\tau)=\int_{\mathbb{R}}e^{it\tau}f(t)dt$ in time gives
\begin{equation*}
P(\tau)\check{w}(\tau, x)=\check{g}(\tau,x),
\end{equation*}
where both sides are holomorphic when $-\Im\tau\in(0,G)$ since both $w$ and $g$ are supported in $\{t\geq0\}$. In particular, for $\tau\in\mathbb{R}$ and $\gamma\in(0,\min(G,s/2))$,
\begin{equation*}
\check{w}(\tau-i\gamma,x)=R(\tau-i\gamma)\check{g}(\tau,x).
\end{equation*}
Taking $H^1$-norm, we have by \eqref{e:resolvent-s-1} and Parseval formula 
\begin{equation*}
\|e^{\gamma t}w\|_{L^2(\mathbb{R};H^1)}\leq C\|e^{\gamma t}g\|_{L^2(\mathbb{R};H^s)}.
\end{equation*}
We notice that $w=1$ and $\supp g\subset\{0\leq t\leq 1\}$, thus 
\begin{equation*}
\|e^{\gamma t}v\|_{L^2([1,\infty);H^1)}\leq C\|g\|_{L^2([0,1];H^s)}
\leq C\left(\|v\|_{L^2([0,1];H^s)}+\|\partial_tv\|_{L^2([0,1];H^s)}\right)
\end{equation*}
where the right-hand side can be estimated by the $\mathcal{H}^s$-norm of the initial data using standard energy estimate. Hence we obtain the following integrated form of \eqref{e:energy-decay}:
\begin{equation}
\label{e:energy-decay-integral}
\|e^{\gamma t}v\|_{L^2([1,\infty);H^1)}\leq C\|(v_0,v_1)\|_{\mathcal{H}^s}.
\end{equation}
To passing from \eqref{e:energy-decay-integral} to \eqref{e:energy-decay}, we can use the following lemma from \cite[\S 5.3]{e-z}, (see \cite[Lemma 9]{schenckpressure}):
\begin{lem}
There exists $C>0$ such that for any solution $v$ of \eqref{e:dampwave},
\begin{equation}
\label{e:energy-intgral}
E(v(T))\leq C\|v\|_{L^2([T-2,T+1],H^1)}^2,\quad T\geq 2.
\end{equation}
\end{lem}
Combining \eqref{e:energy-decay-integral} and \eqref{e:energy-intgral} with $T>3$ finishes the proof of Theorem \ref{t:energy-decay}.

\subsection{Final remarks}
\label{e:remark}

Finally, as pointed out by Kiril Datchev and Jared Wunsch to the author, the result in \cite{measupp} already implies a subexponential decay of the energy. In particular, \cite[Theorem 2]{measupp} shows that for any $a\in C^\infty(M)$ which is not identically zero, there are constants $C,h_0>0$ depending only on $M$ and $a$ such that for any $u\in H^2(M)$ and $0<h<h_0$,
\begin{equation}
\label{e:main-measupp}
\|u\|_{L^2}\leq C\|au\|_{L^2}+\frac{C\log(1/h)}{h}\|(-h^2\Delta-1)u\|_{L^2}.
\end{equation}
We can obtain the following weaker resolvent estimate in a smaller domain (comparing to \eqref{e:reso-bound}) directly from \eqref{e:main-measupp} by a straightforward argument. There exists a constant $C>0$ such that, for $z$ with 
\begin{equation}
\label{e:weak-domain}
|z-1|\leq\frac{h}{C(\log(1/h))^2},
\end{equation}
we have
\begin{equation}
\label{e:weak-resolvent}
\|R(z,h)\|_{L^2\to L^2}\leq\frac{C(\log(1/h))^2}{h}.
\end{equation}
In particular, \cite[Theorem 6.1]{csvw} (with $\alpha(|\lambda|^{-1})=(\log(2+|\lambda|))^2$, $k=2$, $P(r)=(\log(r))^{-2}$ and $F(t)=e^{t^{1/3}/C}$) shows that \eqref{e:weak-resolvent} implies the following statement: For any $s>0$, there exists $C>0$ such that for any $(v_0,v_1)\in \mathcal{H}^s$, the solution to \eqref{e:dampwave} satisfies
\begin{equation}
\label{e:weak-energy-decay}
E(v(t))\leq Ce^{-t^{1/3}/C}\|(v_0,v_1)\|_{\mathcal{H}^s}.
\end{equation}
This is of course weaker than our main result \eqref{e:energy-decay}.

\def\arXiv#1{\href{http://arxiv.org/abs/#1}{arXiv:#1}}


\begin{thebibliography}{0}

\bibitem[An08]{anan} Nalini Anantharaman,
	\emph{Entropy and the localization of eigenfunctions,\/}
	Ann. of Math. \textbf{168}(2008), 435--475.
	
\bibitem[An10]{nalinideviation} Nalini Anantharaman,
	\emph{Spectral deviations for the damped wave equation,\/}
	Geom. Funct. Anal. \textbf{20}(2010), 593--626.
	
\bibitem[AnLe14]{nalinidamped} Nalini Anantharaman and Mathieu L\'eautaud, with an appendix by St\'ephane Nonnenmacher,
	\emph{Sharp polynomial decay rates for the damped wave equation on the torus,\/}
	Anal. PDE \textbf{7}(2014), 159--214.
	
\bibitem[AnNo07]{anno} Nalini Anantharaman and St\'ephane Nonnenmacher,
	\emph{Half-delocalization of eigenfunctions of the Laplacian on an Anosov manifold,\/}
	Ann. Inst. Fourier \textbf{57}(2007), 2465--2523.

\bibitem[BLR92]{blr} Claude Bardos, Gilles Lebeau and Jeffrey Rauch,
	\emph{Sharp sufficient conditions for the observation, control, and stabilization of waves from the boundary,\/}
SIAM J. Control Optim. \textbf{30} (1992), 1024-1065.

\bibitem[BoDy16]{fullgap} Jean Bourgain and Semyon Dyatlov,
	\emph{Spectral gaps without the pressure condition,\/}
	preprint, \arXiv{1612.09040}.
	
\bibitem[BuCh15]{burqchr} Nicolas Burq and Hans Christianson,
	\emph{Imperfect geometric control and overdamping for the damped wave equation,\/}
	Comm. Math. Phys. \textbf{336}(2015), 101--130.
	
\bibitem[BuHi07]{burqhitrik} Nicolas Burq and Michael Hitrik,
	\emph{Energy decay for damped wave equations on partially rectangular domains,\/}
	Math. Res. Lett. \textbf{14}(2007), 35--47.
	
\bibitem[BuZi16]{burqzuily} Nicolas Burq and Claude Zuily,
	\emph{Concentration of Laplace eigenfunctions and stabilization of weakly damped wave equation,\/}
	Comm. Math. Phys. \textbf{345} (2016), 1055--1076.
	
\bibitem[BuZw04]{bz04} Nicolas Burq and Maciej Zworski,
	\emph{Geometric control in the presence of a black box,\/}
	J. Amer. Math. Soc. \textbf{17} (2004), 443-471.
	
\bibitem[Ch07]{hans} Hans Christianson,
	\emph{Semiclassical non--concentration near hyperbolic orbits,\/}
	J. Funct. Anal. \textbf{246}(2007), 145--195.
	
\bibitem[Ch10]{hanscorrection} Hans Christianson,
	\emph{Corrigendum ``Semiclassical non-concentration near hyperbolic orbits''[J. Funct. Anal. \textbf{246}(2007), 145--195],\/}
	J. Funct. Anal. \textbf{258}(2010), 1060--1065.
	
\bibitem[CSVW12]{csvw} Hans Christianson, Emmanuel Schenck, Andras Vasy and Jared Wunsch,
	\emph{From resolvent estimates to damped waves,\/}
	J. Anal. Math. \textbf{122}(2014), 143--162.

\bibitem[DiSj99]{DimassiSjostrand} Mouez Dimassi and Johannes Sj\"ostrand,
	\emph{Spectral asymptotics in the semi-classical limit,\/}
	Cambridge University Press, 1999.

\bibitem[DFG15]{rrh} Semyon Dyatlov, Fr\'ed\'eric Faure, and Colin Guillarmou,
	\emph{Power spectrum of the geodesic flow on hyperbolic manifolds,\/}
	Analysis\&PDE \textbf{8}(2015), 923--1000.
	
\bibitem[DyJi17]{measupp} Semyon Dyatlov and Long Jin,
	\emph{Semiclassical measures on hyperbolic surfaces have full support,\/}
	preprint, \arXiv{1705.05019}.

\bibitem[DyZa16]{hgap} Semyon Dyatlov and Joshua Zahl,
	\emph{Spectral gaps, additive energy, and a fractal uncertainty principle,\/}
	Geom. Funct. Anal. \textbf{26}(2016), 1011--1094.
	
\bibitem[DyZw]{resonance} Semyon Dyatlov and Maciej Zworski,
	\emph{Mathematical theory of scattering resonances,\/}
	book in progress,
	\url{http://math.mit.edu/~dyatlov/res/}

\bibitem[Hi03]{hitrik} Michael Hitrik,
	\emph{Eigenfrequencies and expansions for damped wave equations,\/}
	Methods Appl. Anal. \textbf{10}(2003), 543--564.

\bibitem[LeLe17]{ll} Mathieu L\'eautaud and Nicolas Lerner,
	\emph{Energy decay for a locally undamped wave equation,\/}
	Annales de la facult\'e des sciences de Toulouse \textbf{26}(2017), 157--205.

\bibitem[Le96]{lebeau} Gilles Lebeau,
	\emph{Equation des ondes amorties,\/}
	Algebraic and Geometric Methods in Mathematical Physics (Kaciveli 1993), Math. Phys. Stud. \textbf{19}, Kluwer Acad. Publ. Dordrecht (1996), 73--109.

\bibitem[LiRa05]{liurao} Zhuangyi Liu and Bopeng Rao,
	\emph{Characterization of polynomial decay rate for the solution of linear evolution equation,\/}
	Z. Angew. Math. Phys. \textbf{56}(2005), 630-644.
	
\bibitem[MaMa82]{mama} Aleksandr Semenovich Markus and Vladimir Igorevich Matsaev ,
	\emph{Comparison theorems for spectra of linear operators and spectral asymptotics,\/}
	Trudy. Moskov. Mat. Obshch. \textbf{45}(1982), 133-181.

\bibitem[No11]{nonnenmacher} St\'ephane Nonnenmacher,
	\emph{Spectral theory of damped quantum chaotic systems,\/}
	Journ\'ees EDP, Biarritz (2011).
	
\bibitem[NoZw09-1]{nozw} St\'ephane Nonnenmacher and Maciej Zworski,
	\emph{Quantum decay rates in chaotic scattering,\/}
	Acta Math \textbf{203}(2009), 149--233.
	
\bibitem[NoZw09-2]{nozw2} St\'ephane Nonnenmacher and Maciej Zworski,
	\emph{Semiclassical resolvent estimates in chaotic scattering,\/}
	Applied Mathematics Research eXpress \textbf{1}(2009), 74--86.
	
\bibitem[Ph07]{phung} Kim Dang Phung,
	\emph{Polynomial decay rate for the dissipative wave equations,\/}
	J. Differ. Equ. \textbf{240}(2007), 92--124.
	
\bibitem[RaTa75]{rauchtaylor} Jefferey Rauch and Michael Taylor,
	\emph{Decay of solutions to nondissipative hyperbolic systems on compact manifolds,\/}
	Comm. Pure. Appl. Math. \textbf{28}(1975), 501--523.
	
\bibitem[Re94]{renardy} Michael Renardy, 
	\emph{On the linear stability of hyperbolic PDEs and viscoelastic flows,\/}
	Z. Angew. Math. Phys. \textbf{45}(1994), 854--865.
	
\bibitem[Re14]{nore} Gabriel Rivi\`ere, with an appendix by St\'ephane Nonnenmacher and Gabriel Rivi\`ere,
	\emph{Eigenmodes of the damped wave equation and small hyperbolic subsets,\/}
	Ann. Inst. Fourier (Grenoble) \textbf{64}(2014), 1229-1267.
	
\bibitem[Sc10]{schenckpressure} Emmanuel Schenck,
	\emph{Energy decay for the damped wave equation under a pressure condition,\/}
	Comm. Math. Phys. \textbf{300}(2010), 375--410.

\bibitem[Sc11]{schenck} Emmanuel Schenck,
	\emph{Exponential stabilization without geometric control,\/}
	Math. Res. Lett. \textbf{18}(2011), 379--388.
	
\bibitem[Sj00]{sjostrand} Johannes Sj\"ostrand,
	\emph{Asymptotic distribution of eigenfrequencies for damped wave equations,\/}
	Publ. RIMS. Kyoto Univ. \textbf{36}(2000), 573--611.
	
\bibitem[Zw12]{e-z} Maciej Zworski,
    \emph{Semiclassical analysis,\/}
    Graduate Studies in Mathematics \textbf{138}, AMS, 2012.

\end{thebibliography}
\end{document}